\documentclass[12pt]{article}

\usepackage{amsfonts,amssymb,amsbsy,latexsym,
amsmath,tabulary,graphicx,caption,
fancyhdr,yfonts,hyperref}

\usepackage{libertine}
\usepackage{bold-extra}
\usepackage{textcomp}
\usepackage{pifont}

\usepackage{draftwatermark}
\SetWatermarkAngle{70}
\SetWatermarkLightness{0.925}
\SetWatermarkScale{.5}
\SetWatermarkText{\hspace{-.25em}\copyright {\tt 2020} {\sc Saeed Salehi}}

\newtheorem{theorem}{\textsc{\textbf Theorem}}[section]
\newtheorem{definition}[theorem]{\textsc{\textbf Definition}}

\newtheorem{lemma}[theorem]{\textsc{\textbf Lemma}}

\newtheorem{remark}[theorem]{\textsc{\textbf Remark}}

\newenvironment{proof}{\mbox{\bf Proof}:}
{\hfill\mbox{\ding{113}}\bigskip}

\makeatletter

\AtBeginDocument{
\expandafter\ifx\csname eqalign\endcsname\relax
\def\eqalign#1{\null\vcenter{\def\\{\cr}\openup\jot\m@th
  \ialign{\strut$\displaystyle{##}$\hfil&$\displaystyle{{}##}$\hfil
      \crcr#1\crcr}}\,}
\fi
}

\let\lt=<
\let\gt=>
\def\processVert{\ifmmode|\else\textbar\fi}

\@ifundefined{subparagraph}{
\def\subparagraph{\@startsection{paragraph}{5}{2\parindent}{0ex plus 0.1ex minus 0.1ex}%
{0ex}{\normalfont\small\itshape}}%
}{}

\newcommand\role[1]{\unskip}
\newcommand\aucollab[1]{\unskip}

\@ifundefined{tsGraphicsScaleX}{\gdef\tsGraphicsScaleX{1}}{}
\@ifundefined{tsGraphicsScaleY}{\gdef\tsGraphicsScaleY{.9}}{}
\def\checkGraphicsWidth{\ifdim\Gin@nat@width>\textwidth
	\tsGraphicsScaleX\textwidth\else\Gin@nat@width\fi}

\def\checkGraphicsHeight{\ifdim\Gin@nat@height>.9\textheight
	\tsGraphicsScaleY\textheight\else\Gin@nat@height\fi}

\def\fixFloatSize#1{\@ifundefined{processdelayedfloats}{\setbox0=\hbox{\includegraphics{#1}}\ifnum\wd0<\columnwidth\relax\renewenvironment{figure*}{\begin{figure}}{\end{figure}}\fi}{}}
\let\ts@includegraphics\includegraphics

\def\inlinegraphic[#1]#2{{\edef\@tempa{#1}\edef\baseline@shift{\ifx\@tempa\@empty0\else#1\fi}\edef\tempZ{\the\numexpr(\numexpr(\baseline@shift*\f@size/100))}\protect\raisebox{\tempZ pt}{\ts@includegraphics{#2}}}}

\AtBeginDocument{\def\includegraphics{\@ifnextchar[{\ts@includegraphics}{\ts@includegraphics[width=\checkGraphicsWidth,height=\checkGraphicsHeight,keepaspectratio]}}}

\def\URL#1#2{\@ifundefined{href}{#2}{\href{#1}{#2}}}

\def\UrlOrds{\do\*\do\-\do\~\do\'\do\"\do\-}%
\g@addto@macro{\UrlBreaks}{\UrlOrds}
\makeatother

\makeatletter

\def\wileyIndent{1pt}
\usepackage[paperheight=10in,paperwidth=6.5in,margin=2cm,headsep=.5cm,top=2.5cm]{geometry}

\renewenvironment{abstract}
{\vspace*{-1pc}\trivlist\item[]\leftskip\wileyIndent\hrulefill\par\vskip4pt\noindent\textbf{\abstractname}\mbox{\null}\\}{\par\noindent\hrulefill\endtrivlist}

\def\author#1{\gdef\@author{\hskip-\dimexpr(\tabcolsep)\hskip\wileyIndent\parbox{\dimexpr\textwidth-\wileyIndent}{\centering\bfseries#1}}}

\def\title#1{\gdef\@title{\centering\bfseries\ifx\@articleType\@empty\else\@articleType\\\fi#1}}

\let\@articleType\@empty \def\articletype#1{\gdef\@articleType{{\normalfont\itshape#1}}}

\fancypagestyle{headings}{\fancyhf{}\fancyhead[C]{\RunningHead}\fancyhead[R]{\thepage}}

 \def\audegree#1{}

\makeatother

\usepackage[T1]{fontenc}
\makeatother
\usepackage[numbers,sort&compress]{natbib}

\begin{document}

\title{On the Notions of {\em Rudimentarity, Primitive Recursivity} and {\em Representability} of Functions and Relations}

\author{\textsc{Saeed Salehi}}

\date{\sf 1 March 2020}

\def\RunningHead{\!\!\!\!\!{\sc S. Salehi}: \; {\em  Rudimentarity, Primitive Recursivity \& Representability}}
\def\RunningAuthor{\textsc{Saeed Salehi}}

\maketitle

\vspace{-2em}
\begin{abstract}
It is quite well-known from Kurt G\"odel's (1931) ground-breaking result on the  Incompleteness Theorem that  rudimentary relations (i.e., those definable by bounded formulae) are primitive recursive, and that primitive recursive functions are representable in sufficiently strong arithmetical theories. It is also known, though perhaps not as well-known as the former one, that some primitive recursive relations are not  rudimentary. We present a simple and elementary proof of this fact in the first part of the paper. In the second part, we review some possible notions of representability of functions studied in the literature, and give a new proof of the equivalence of the weak representability with the (strong) representability of functions in sufficiently strong arithmetical theories. Our results shed some new light on the notions of rudimentary, primitive recursive, and representable functions and relations, and clarify, hopefully, some misunderstandings and confusing errors in the literature.

\medskip

$\cdot$
03D20  	
$\cdot$
03F30.


\end{abstract}

\section{Introduction and Preliminaries}\label{sec:intro}
Primitive recursive functions are what were called ``rekursiv'' by Kurt G\"odel in his seminal 1931 paper \cite{feferman,godel} where he proved the celebrated incompleteness theorem. The main features of the primitive recursive functions used by G\"odel  were the following:
\begin{enumerate}
\item They are {\em computable} (i.e., for each primitive recursive function there exists an algorithm that computes it). However, we now know that they do not make up the whole (intuitively) computable functions (from tuples of natural numbers to natural numbers, $\mathbb{N}^k\rightarrow\mathbb{N}$). So,   ``rekursiv'' functions are now called ``primitive recursive'' functions, which constitute a sub-class  of  {\sl recursive functions} that 
     are believed to constitute the whole computable functions.
\item They are {\em representable} in (sufficiently  expressive  and sufficiently strong) formal arithmetical theories. It is now known that, more generally, (only) recursive functions   are representable in (all the) recursively enumerable, sufficiently strong and sufficiently expressive theories (see Section~\ref{sec:rep} below).
\item Theories whose {\em set of axioms} are primitive recursive  and extend a base theory (such as  Robinson's Arithmetic $\textsf{\textsl{Q}}$), are incomplete. It was later found out that this holds more generally for recursively enumerable extensions of $\textsf{\textsl{Q}}$. Also, by William Craig's Trick, every such theory  is equivalent with another theory whose set of axioms is rudimentary (i.e., definable by a bounded formula).
\end{enumerate}
Even though  one can set up the whole theory of computable functions (aka recursion theory) and the incompleteness theorems without introducing  the notion of primitive recursive functions (and relations),  the theory of primitive recursive functions  is a main topic in the literature on recursive function theory and the incompleteness theorems. For the sake of completeness we review some basic notions of this theory.

\begin{definition}[Primitive Recursive  Functions]\label{def:prf}{\rm

\

\noindent
The class of {\em primitive recursive} ({\sc pr}) functions is the smallest class that contains the initial functions

(i) the constant zero function $\boldsymbol\zeta\colon\mathbb{N}\rightarrow\mathbb{N}, \,  \boldsymbol\zeta(x)\!=\!0$,

(ii) the successor function $\boldsymbol\sigma\colon\mathbb{N}\rightarrow\mathbb{N}, \, \boldsymbol\sigma(x)\!=\!x\!+\!1$, and

(iii) the projection functions $\boldsymbol\pi_i^n\colon\mathbb{N}^n\rightarrow\mathbb{N}, \, \boldsymbol\pi_i^n(x_1,\cdots,x_n)\!=\!x_i$, for

\quad any $n\geqslant 1$ where $i\in\{1,\cdots,n\}$;

\noindent and is closed under

(I)  composition of functions, i.e., contains the function $h\colon\mathbb{N}^n\rightarrow\mathbb{N}$, if it already contains the functions $g_1,\cdots,g_m\colon\mathbb{N}^n\rightarrow\mathbb{N}$ and $f\colon\mathbb{N}^m\rightarrow\mathbb{N}$, where  $h(x_1,\cdots,x_n)\!=\!f\big(
g_1(x_1,\ldots,x_n),\cdots,g_m(x_1,\cdots,x_n)\big)$, and

(II)  primitive recursion, i.e., contains  the function  $h\colon\mathbb{N}^{n+1}\rightarrow\mathbb{N}$, if it already contains the functions $f\colon\mathbb{N}^{n}\rightarrow\mathbb{N}$ and $g\colon\mathbb{N}^{n+2}\rightarrow\mathbb{N}$, where $h\colon\mathbb{N}^{n+1}\rightarrow\mathbb{N}$ is inductively defined by

$\begin{cases}
h(x_1,\cdots,x_n,0)\!=\!f(x_1,\cdots,x_n), &  \\ h(x_1,\cdots,x_n,x\!+\!1)\!=\!g\big(h(x_1,\cdots,x_n,x),x_1,\cdots,x_n,x\big). &
\end{cases}$
}\end{definition}

\vspace{-2.25em}
\hfill\ding{71}

\vspace{1em}

\noindent 
By Hermann Grassmann's recursive definition of addition and multiplication, it can be shown that these functions  ($+\!\colon\!\mathbb{N}^2\rightarrow\mathbb{N}, (x,y)\mapsto x\!+\!y$ and  $\times\!\colon\!\mathbb{N}^2\rightarrow\mathbb{N}, (x,y)\mapsto x\!\cdot\!y$) are {\sc pr}; so are the following sign functions are primitive recursive:

${\sf sg}(x)=\begin{cases}0 &  \mbox{ if  }\ x\!=\!0, \\ 1 & \mbox{ if  }\ x\!\neq\!0, \end{cases}$ \qquad and
\qquad $\widetilde{\sf sg}(x)=\begin{cases}1 &  \mbox{ if  }\ x\!=\!0, \\ 0& \mbox{ if  }\ x\!\neq\!0. \end{cases}$

\bigskip

\begin{definition}[Primitive Recursive  Relations]\label{def:prr}{\rm

\

\noindent
The {\em characteristic function} of a relation $R\!\subseteq\!\mathbb{N}^n$  is $\boldsymbol\chi_R\colon\mathbb{N}^n\rightarrow\{0,1\}$,
$\boldsymbol\chi_R(x_1,\cdots,x_n) = \begin{cases} 1 & \mbox{ if } (x_1,\cdots,x_n)\in R, \\ 0 & \mbox{ if } (x_1,\cdots,x_n)\not\in R. \end{cases}$

\noindent A relation is called {\em primitive recursive} ({\sc pr}),  if its characteristic function is primitive recursive.
}\hfill\ding{71} \end{definition}

For example, the equality ($=$) and inequality ($\leqslant$)   can be shown to be {\sc pr} relations. The following identities show that the class of {\sc pr}  relations is closed under Boolean operations and bounded quantifications:

$\boldsymbol\chi_{R\cap S}\!=\!\boldsymbol\chi_R\cdot\boldsymbol\chi_S; \quad
\boldsymbol\chi_{R^\complement}\!=\!
\widetilde{\sf sg}(\boldsymbol\chi_R);
 \quad \boldsymbol\chi_{R\cup S}\!=\!{\sf sg}(\boldsymbol\chi_R\!+\!\boldsymbol\chi_S);$

\bigskip

$\begin{cases}
\boldsymbol\chi_{\forall x\leqslant\alpha R(\vec{z},x)}(\vec{z},0)\!=\!\boldsymbol\chi_R(\vec{z},0), &   \\
\boldsymbol\chi_{\forall x\leqslant\alpha R(\vec{z},x)}(\vec{z},\alpha\!+\!1)\!=
\!\boldsymbol\chi_{\forall x\leqslant\alpha R(\vec{z},x)}(\vec{z},\alpha)
\cdot\boldsymbol\chi_R(\vec{z},\alpha\!+\!1); &
\end{cases}$

$\begin{cases}
\boldsymbol\chi_{\exists x\leqslant\alpha R(\vec{z},x)}(\vec{z},0)\!=\!
\boldsymbol\chi_R(\vec{z},0), &   \\
\boldsymbol\chi_{\exists x\leqslant\alpha R(\vec{z},x)}(\vec{z},\alpha\!+\!1)\!=\!
{\sf sg}\big(\boldsymbol\chi_{\forall x\leqslant\alpha\,P(\vec{z},x)}(\vec{z},\alpha)\!+\!
\boldsymbol\chi_P(\vec{z},\alpha\!+\!1)\big). &
\end{cases}$

\begin{definition}[Rudimentary  Relations]\label{def:delta0}{\rm

\

\noindent
A formula in the language of arithmetic $\langle 0,1,+,\times,\leqslant\rangle$ is called {\em bounded}, if it has been constructed from atomic formulas (of the form $t\!=\!s$ or $t\!\leqslant\!s$, for terms $s,t$) by means of negation, conjunction, disjunction, implication, and bounded quantifications (of the form $\forall x\!\leqslant\!t$ or $\exists x\!\leqslant\!t$, where the formula $\forall x\!\leqslant\!t\,\mathcal{A}(x,t)$ abbreviates
$\forall x \big[x\!\leqslant\!t\rightarrow \mathcal{A}(x,t)\big]$ and
$\exists x\!\leqslant\!t\,\mathcal{A}(x,t)$ is an abbreviation for  $\exists x \big[x\!\leqslant\!t\wedge \mathcal{A}(x,t)\big]$ for term $t$ and variable $x$ which is not free in $t$).

The class of bounded formulas is denoted by $\Delta_0$.

\noindent A relation $R\subseteq\mathbb{N}^n$ is called {\em rudimentary} or {\em bounded definable}, or simply $\Delta_0$,  if it can be defined by a $\Delta_0$-formula, i.e., there exists a $\Delta_0$-formula $\varphi(x_1,\cdots,x_n)$ such that
$R=\{(x_1,\cdots,x_n)\mid\mathbb{N}\models\varphi(x_1,\cdots,x_n)\}$.
}\hfill\ding{71} \end{definition}

The above arguments show that all the $\Delta_0$ relations are {\sc pr}; see also e.g.  \cite{cook,hedman,rautenberg}.
The question as to whether the converse holds, i.e., {\sl whether  every {\sc pr} relation is $\Delta_0$}, has been mentioned in very few  places. Unfortunately, as will be indicated,  some of them are  wrong or misleading:

\begin{itemize}
\item[(1)]   On page 315 of \cite{hedman} we read: ``A relation is primitive recursive if and only if it is definable by a $\Delta_0$ formula.
We presently prove one direction of this fact. The other direction shall become
apparent after Section 8.3 of the next chapter and is left as Exercise 8.6.''

This leaves the reader wondering what (theorems or techniques) will be provided in Chapter 8 (the incompleteness theorems) of the book~\cite{hedman} that will enable the reader to show that every {\sc pr} relation is rudimentary, i.e., $\Delta_0$ definable. The fact of the matter is that, as will be seen below, {\sl it is not true that  every {\sc pr} relation is $\Delta_0$}.

\item[(2)]
On page 239 of \cite{rautenberg} we read as {\bf Remark 1}, ``Induction on the
$\Delta_0$-formulas readily shows that all $\Delta_0$-predicates are p.r.  The converse does not
hold; an example is the graph of the very rapidly growing hyperexponentiation,
recursively defined by ${\rm hex}(a, 0)\!=\!1$ and ${\rm hex}(a,Sb)\!=\!a^{{\rm hex}(a,b)}$.''

The {\em graph} of a function $f\colon X\rightarrow Y$ is, by definition,  the relation $$\Gamma_f\!=\!\{(x,y)\in X\!\times\!Y\mid y\!=\!f(x)\}.$$
Let us note that the graph $\Gamma_f$ of a {\sc pr} function $f$ is a {\sc pr} relation, since $\boldsymbol\chi_{\Gamma_f}(\vec{a},b)\!=\!
\boldsymbol\chi_=(f(\vec{a}),b)$.
Now,  ${\rm hex}$ is a {\sc pr}  function, and so its graph is a {\sc pr} relation. But the claim that this relation is not $\Delta_0$ has not been proved in~\cite{rautenberg}. In fact, it has been shown in~\cite{calude} (see also~\cite{esbelinmore}) that this is not true: {\sl the graph of  ${\rm hex}$   is actually $\Delta_0$.}

\item[(3)] We read in the Abstract of \cite{esbelinmore},   ``The question of whether a given primitive recursive relation is rudimentary is in some cases difficult and related to several well-known open  questions in theoretical computer science''. Also, on page 130 of \cite{esbelinmore} we read, ``However, it is difficult to exhibit a {\em natural} arithmetical relation which can be proved not to be rudimentary'' (emphasize in the original) and that ``This paper is an attempt to systemize ... proving that various primitive recursive relations are rudimentary''. Later, on page 132 we read,  ``Hence, the main  way  of exhibiting a primitive recursive relation which is not rudimentary is to choose it in $\textswab{C}^3_\ast\setminus\textswab{C}^2_\ast$. Although it is true that infinitely many [such] relations exist, we know no natural example''.
    Here, by ``{\it natural}'' the authors mean  a relation ($\subseteq\mathbb{N}^k$) that the number-theorists use and work with.

\item[(4)] On page 85 of \cite{cook} after proving   that  ``Every $\Delta_0$ relation is primitive recursive'' as a {\bf Lemma}, we read,
``{\bf Remark}: The converse of the above lemma is false, as can be shown by a diagonal argument.
For those familiar with complexity theory, we can clarify things as follows. As
noted in the Side Remark above, all $\Delta_0$ relations can be recognized in linear space on a
Turing machine. On the other hand, it follows from the Ritchie-Cobham Theorem that all
relations recognizable in space bounded by a primitive recursive function of the input length
are primitive recursive. In particular, space ${\rm O}(n^2)$ relations are primitive recursive, and a
straightforward diagonal argument shows that there are relations recognizable in $n^2$ space
which are not recognizable in linear space, and hence are not $\Delta_0$ relations.''

The mentioned side-remark (that ``All $\Delta_0$ relations can be recognized in linear space on a Turing machine,
when input numbers are represented in binary notation'') has not been proved in  \cite{cook}. This was proved first by J. R.  Myhill in \cite{myhill}.
\end{itemize}

So, there should exist some {\sc pr} relations that are not $\Delta_0$.     In Section~\ref{sec:pr} we will show that a specific {\sc pr} relation is not $\Delta_0$, by a carefully detailed proof with little background in complexity theory or formal arithmetics. This relation may not look natural for number-theorists, but is sufficiently natural for logicians.

In the second part, Section \ref{sec:rep}, we will study some possible notions of representability of functions and relations in arithmetical theories and will compare  their strength with each other; we will provide a new proof for an old theorem which appears in a very few places with a much longer  proof. The theorem says that {\sl every weakly representable function is (strongly) representable}; this is usually proved by showing that ($\texttt{A}$) every weakly representable function is recursive, and ($\texttt{B}$) every recursive function is (strongly) representable. Our proof is direct  and  elementary.

\section{Rudimentarity vs. Primitive Recursivity}\label{sec:pr}
Let us  be given a fixed G\"odel coding $\alpha\mapsto\ulcorner\!\alpha\!\urcorner$, which is primitive recursive (as is usually presented in the literature).
Our example of a {\sc pr} relation that is not $\Delta_0$, uses an idea of Alfred Tarski; that the truth relation of arithmetical sentences is not arithmetically definable. Likewise, the truth of $\Delta_0$-sentences is not $\Delta_0$; but, as will be shown later, it is {\sc pr}.

\begin{definition}[$\Delta_0$-Satisfaction]
\label{def:sat}{\rm

\

\noindent
Let ${\tt Sat}_{\Delta_0}$ be the set of all the ordered pairs $(\ulcorner\!\theta(\vec{\nu})\!\urcorner,a)$, where $\theta(\vec{\nu})$ is a $\Delta_0$-formula with the shown free variables and $a\!\in\!\mathbb{N}$, such that $\mathbb{N}\vDash\theta(\vec{a})$; i.e., the sentence  resulted from  substituting $a$ for every free variable of $\theta$ is true (in the standard model of natural numbers).
}\hfill\ding{71} \end{definition}

In the other words, ${\tt Sat}_{\Delta_0}=
\{(\ulcorner\!\theta(\vec{\nu})\!\urcorner,a)
\mid \mathbb{N}\!\vDash\!\theta(\vec{a}) \; \& \;  \theta\!\in\!\Delta_0\}.$

\begin{theorem}[Non-Rudimentarity of $\Delta_0$-Satisfaction]\label{th:delta0}

\

\noindent
The relation ${\tt Sat}_{\Delta_0}(x,y)$ is not definable by any $\Delta_0$-formula.
\end{theorem}
\begin{proof}

\noindent
If a $\Delta_0$-formula such as  $\boldsymbol\varsigma(x,y)$ defines  the relation
 ${\tt Sat}_{\Delta_0}$, then for the formula $\theta(x)\!=\!\boldsymbol\neg\boldsymbol\varsigma(x,x)$
 (which is $\Delta_0$) and number $m\!=\!\ulcorner\!\theta(x)\!\urcorner$, we have $\mathbb{N}\vDash\theta(m)\!\equiv\!{\tt Sat}_{\Delta_0}(\ulcorner\theta(x)\urcorner,m)
\!\equiv\!\boldsymbol\varsigma(m,m)
\!\equiv\!\boldsymbol\neg\theta(m)$, a  contradiction!
\end{proof}

In the rest of this section, we show that ${\tt Sat}_{\Delta_0}$ is a {\sc pr} relation.
This can already be inferred from the results of \cite{lessan}; see also \cite[Definition 4.1.3 and Lemma 4.1.4]{weak} and \cite[Theorem 2]{parisdimitra} and \cite[Corollary 5.5]{hp}. All of them use advanced arguments that cannot be mentioned in more elementary texts like \cite{cook,hedman,rautenberg}. Our aim here is to provide an elementary proof for primitive recursivity of
${\tt Sat}_{\Delta_0}$ in such a way that  it can be used, along with Theorem~\ref{th:delta0},  in textbooks for clarifying the status of {\sc pr} vs.  $\Delta_0$ relations.
\begin{remark}[On G\"odel Coding]\label{rem:coding}{\rm

\

\noindent
We can assume that the set of the G\"odel codes of the variables is definable by a $\Delta_0$-formula; for example we can keep even numbers $2,4,6,\cdots$ for coding the variables $v_0,v_1,v_2,\cdots$ respectively, and then code the rest of the language (propositional connectives, quantifiers, parentheses and function and relation symbols) by odd numbers. As a result of this way of coding,    ${\tt var}(x)\!\equiv\!\exists y\!\leqslant\!x\,(y\!=\!2x\!+\!2)$  is a $\Delta_0$-formula that defines the variables. Other syntactical notions of {\em terms, formulas, sentences, bounded sentences, proofs}, etc. can be shown to be {\sc pr} as usual (see e.g. \cite{hp,hedman,mendelson,rautenberg}).
Let $\mathfrak{p}_0,\mathfrak{p}_1,\mathfrak{p}_2,\cdots$ be the sequence of all prime numbers ($2,3,5,\cdots$). Let us code the sequence $\langle\alpha_0,\alpha_1,\cdots,\alpha_k\rangle$ by the number $\prod_{i\leqslant k}\mathfrak{p}_i^{\alpha_i+1}$.  Let us note that this way, the code of any such sequence will be non-greater than $\mathfrak{p}_k^{kA}$, where $A$ is any number greater than all $\alpha_i$'s. Also let us recall that the functions $i\mapsto\mathfrak{p}_i$  and $(k,A)\mapsto\mathfrak{p}_k^{kA}$ are both {\sc pr}  (see e.g. \cite{hedman,mendelson,rautenberg}).
}\hfill\ding{71}
\end{remark}

\begin{definition}[Terms, Bounded Formulas, Valuations, etc.]\label{def:prthings}{\rm

\

\noindent
For a fixed G\"odel coding, let the relation
\begin{itemize}
\item ${\tt var}(x)$ hold, when ``$x$ is (the G\"odel code of) a {\em variable}''.
\item ${\tt trm}(x)$ hold, when ``$x$ is (the G\"odel code of) a {\em term}''.
\item ${\tt atm}(x)$ hold, when ``$x$ is (the G\"odel code of) an {\em atomic formula}''.
\item ${\tt fml}_{\Delta_0}(x)$ hold, when ``$x$ is (the G\"odel code of) a {\em $\Delta_0$-formula}''.
\item ${\tt val}(x,y,z)$ hold, when ``$x$ is (the G\"odel code of) a term with the free variables $\langle\nu_0,\cdots,\nu_\ell\rangle$, $y$ is (the G\"odel code of) a sequence of numbers $\langle a_0,\cdots,a_\ell\rangle$, and $z$ is the {\em value} of the term $x$ when each $\nu_i$ is substituted with $a_i$, for $i\!\leqslant\!\ell$''.
\hfill\ding{71}
\end{itemize}
} \end{definition}

\begin{lemma}[${\tt var}$, ${\tt trm}$, ${\tt fml}_{\Delta_0}$ and ${\tt val}$ are {\sc pr}]\label{lem:allpr}

\

\noindent
The relations ${\tt var}$, ${\tt trm}$, ${\tt atm}$, ${\tt fml}_{\Delta_0}$ and ${\tt val}$ are {\sc pr}.
\end{lemma}
\begin{proof}

\noindent
We already noted (in Remark~\ref{rem:coding}) that the ${\tt var}$ relation can even be $\Delta_0$   (and so it is a {\sc pr} relation)  by a modest convention on coding. There is also a $\Delta_0$ relation ${\tt seq}(x)$ which holds of $x$ when $x$ is (the G\"odel code of) a sequence. Let $\ell en(x)$ denote the length of $x$ and $[x]_i$, for each $i\!<\!\ell en(x)$, denote the $i$-th element of $x$. Thus, if ${\tt seq}(x)$ holds, then $x$ codes the sequence $\langle [x]_0,[x]_1,\cdots,[x]_{\ell en(x)-1}\rangle$. Let us recall that   $x\mapsto\ell en(x)$ and $(i,x)\mapsto [x]_i$ are both {\sc pr}
functions. Let $y\!=\!\ell ast(x)$ abbreviate $y\!=\![x]_{\ell en(x)-1}$.

\noindent $^{_\bullet}$ Let ${\tt trmseq}(x)$ be the following $\Delta_0$ relation:

${\tt seq}(x)\wedge\forall i\!<\!\ell en(x)\Big[ [x]_i\!=\!\ulcorner 0 \urcorner\vee
[x]_i\!=\!\ulcorner 1 \urcorner \vee {\tt var}([x]_i) \; \vee$

\hfill $\exists j,k\!<\!i\Big([x]_i\!=\!\ulcorner\!([x]_j\!+\![x]_k)\!\urcorner \vee
[x]_i\!=\ulcorner\!([x]_j\!\times\![x]_k)\!\urcorner\Big)\Big]$.

\noindent Now, ${\tt trm}(x)$ can be written  as $\exists s\!\leqslant\! \mathfrak{p}_x^{(x+1)^2}{\tt trmseq}(s)\wedge\ell ast(s)\!=\!x$; noting that the building sequence of a term $x$ has length at most $x$ and all the elements of that sequence are non-greater than $x$. So, ${\tt trm}(x)$ is  {\sc pr}.

\noindent $^{_\bullet}$  That ${\tt atm}(x)$ is a {\sc pr} relation, follows from the following:

   ${\tt atm}(x)\!\equiv\!\exists u,v\!<\!x\big[{\tt trm}(u)\wedge{\tt trm}(v)\wedge \big(x\!=\!\ulcorner\!(u\!=\!v)\!\urcorner \vee
x\!=\!\ulcorner\!(u\!\leqslant\!v)\!\urcorner\big)\big]$.

\noindent $^{_\bullet}$  Without loss of generality we can assume that the propositional connectives are only $\neg$ and $\rightarrow$ and the only quantifier is $\forall$. Now, the following $\Delta_0$-formula defines the building sequence of a bounded formula:

${\tt fml}_{\Delta_0}{\tt seq}(x)\!\equiv\!{\tt seq}(x)\wedge\forall i\!<\!\ell en(x)\Big[ {\tt atm}([x]_i) \vee$

\hspace{6.25em} $\exists j,k\!<\!i\Big(
[x]_i\!=\!\ulcorner\!(\neg[x]_j)\!\urcorner \vee
[x]_i\!=\!\ulcorner\!([x]_j\!\rightarrow\![x]_k)\!\urcorner \vee$

\hspace{6.25em} $\exists v,t\!<\!x\big[{\tt var}(v)\wedge {\tt trm}(t)\wedge
[x]_i\!=\!\ulcorner\!(\forall v\!\leqslant\!t)[x]_j\!\urcorner
\big]\Big)\Big]$.

\noindent
So, ${\tt fml}_{\Delta_0}(x)\!\equiv\!\exists s\!\leqslant\!\mathfrak{p}_x^{(x+1)^2}{\tt fml}_{\Delta_0}{\tt seq}(s)\wedge\ell ast(s)\!=\!x$ is a {\sc pr} relation.

\noindent $^{_\bullet}$  Let ${\tt valseq}(y,s,t)$ be the following $\Delta_0$ relation:

\noindent
${\tt seq}(y)\wedge{\tt termseq}(s)\wedge{\tt seq}(t)\wedge\ell en(t)\!=\!\ell en(s)\,\wedge \forall i\!<\!\ell en(s) \Big[$

\noindent
$\big([s]_i\!=\!\ulcorner 0 \urcorner \wedge [t]_i\!=\!0\big)\vee \big([s]_i\!=\!\ulcorner 1 \urcorner \wedge [t]_i\!=\!1\big)\vee\big({\tt var}([s]_i)\wedge [t]_i\!=\![y]_i\big)\vee$

\noindent
$\exists j,k\!<\!i\big[
\big([s]_i\!=\!\ulcorner\!([s]_j\!+\![s]_k)\!\urcorner
\wedge [t]_i\!=\![t]_j\!+\![t]_k\big)
\vee$

\hfill $
\big([s]_i\!=\!\ulcorner\!([s]_j\!\times\![s]_k)\!\urcorner
\wedge [t]_i\!=\![t]_j\!\cdot\![t]_k\big)
\big]\Big],$

\noindent
which states that $y,t$ are (the G\"odel code of)   sequences (of numbers) and $s$ is (the G\"odel code of)   a building sequence of a term such that $t$ is the result of substituting the variables of $s$ with the corresponding elements of $y$.
Finally, ${\tt val}(x,y,z)$ is {\sc pr} since it is equivalent with

$\exists s\!\leqslant\!\mathfrak{p}_x^{(x+1)^2}\exists t\!\leqslant\!\mathfrak{p}_z^{(z+1)^2}  {\tt valseq}(y,s,t)\wedge \ell ast(s)\!=\!x\wedge\ell ast(t)\!=\!z$.
\end{proof}

\begin{remark}[${\tt Sat}_{\Delta_0}$ In the Border of {\sc pr} and $\Delta_0$]\label{rem:rem}{\rm

\

\noindent
The main idea of the proofs of Lemma~\ref{lem:allpr} and Theorem~\ref{th:prthings} are from \cite[Chapter 9]{kay}.  Actually, by the techniques of \cite[Chapter V]{hp} one can show that all the relations ${\tt var}(x)$, ${\tt trm}(x)$, ${\tt atm}(x)$, ${\tt fml}_{\Delta_0}(x)$ and ${\tt val}(x,y,z)$ can be $\Delta_0$, under a suitable G\"odel coding. In Theorem~\ref{th:prthings} we will show that ${\tt Sat}_{\Delta_0}(x,y)$ is a {\sc pr} relation, which, by Theorem~\ref{th:delta0}, cannot be $\Delta_0$ under any G\"odel coding. We will see in the proof of Theorem~\ref{th:prthings} that ${\tt Sat}_{\Delta_0}$ is definable by the relations ${\tt var}$, ${\tt trm}$, ${\tt atm}$, ${\tt fml}_{\Delta_0}$ and ${\tt val}$. So, we have a boundary  result here: the {\sc pr} relations ${\tt var}(x)$, ${\tt trm}(x)$, ${\tt atm}(x)$, ${\tt fml}_{\Delta_0}(x)$ and ${\tt val}(x,y,z)$  all can be    $\Delta_0$ under some G\"odel coding, while the {\sc pr} relation ${\tt Sat}_{\Delta_0}(x,y)$    {\em can never be $\Delta_0$}.
}\hfill\ding{71}
\end{remark}

\begin{theorem}[${\tt Sat}_{\Delta_0}$ is a {\sc pr} Relation]\label{th:prthings}

\

\noindent
The relation ${\tt Sat}_{\Delta_0}(x,y)$ is  {\sc pr}.
\end{theorem}
\begin{proof}

\noindent
Define the relation ${\tt sat}_{\Delta_0}{\tt seq}(s,t)$ by  ``$s$ is a building sequence of a $\Delta_0$-formula, and $t$ is a sequence of triples $\langle i,z,w\rangle$ in which  $i\!<\!\ell en(s)$  and $w\!\leqslant\!1$ is a truth value ($1$ for truth and $0$ for falsity) of the formula $[s]_i$ when the variables $v_0,v_1,\cdots$ are interpreted by $[z]_0,[z]_1,\cdots$ respectively''. Let  $z[r/k]$ denote  the sequence resulted from $z$ by substituting its $k$-th element with $r$. The function $z,r,k\mapsto z[r/k]$ is {\sc pr}, and when ${\tt val}(u,z,x)$ holds, then we can have ${\tt val}(u,z,x)$ for some $x\!\leqslant\!\mathfrak{p}_u^{z^u+1}$, since the value of a term $u$ when its free variables are substituted by the elements of $z$ is non-greater than $\mathfrak{p}_u^{z^u+1}$.
Now, ${\tt sat}_{\Delta_0}{\tt seq}(s,t)$ is {\sc pr} since it is defined by:

\noindent
${\tt fml}_{\Delta_0}{\tt seq}(s)\wedge {\tt seq}(t)\wedge \forall l\!<\!\ell en(t)\,\exists i,z,w\!\leqslant\!t$

\centerline{
$\Big[[t]_l\!=\!\langle i,z,w\rangle \wedge i\!<\!\ell en(s) \wedge w\!\leqslant\!1\,\wedge$}

$\Big(\big[\exists u,v\!\leqslant\!s\big({\tt trm}(u)\wedge{\tt trm}(v)\wedge [s]_i\!=\!\ulcorner\!(u\!=\!v)\!\urcorner\,\wedge$

\;
$[w=1\leftrightarrow\exists x\!\leqslant\!\mathfrak{p}_{u+v}^{(z^{u+v}+1)^2}
{\tt val}(u,z,x)\wedge{\tt val}(v,z,x)]\big)\big] \vee $

$\big[\exists u,v\!\leqslant\!s\big({\tt trm}(u)\wedge{\tt trm}(v)\wedge [s]_i\!=\!\ulcorner\!(u\!\leqslant\!v)\!\urcorner\,\wedge$

\;
$[w=1\leftrightarrow\exists x,y\!\leqslant\!\mathfrak{p}_{u+v}^{(z^{u+v}+1)^2}
{\tt val}(u,z,x)\wedge{\tt val}(v,z,y)\wedge x\!\leqslant\!y]\big)\big] \vee $

$\big[\exists j\!<\!i\big([s]_i\!=\!\ulcorner\!(\neg[s]_j)\!\urcorner\,\wedge\exists p\!<\!l\exists w'\!\leqslant\!1
([t]_p\!=\!\langle j,z,w'\rangle\;\wedge$

 \;
$[w\!=\!1\leftrightarrow w'\!=\!0])\big)\big] \vee $

$\big[\exists j,k\!<\!i\big([s]_i\!=\!\ulcorner\!([s]_j\rightarrow [s]_k)\!\urcorner\,\wedge\exists p,q\!<\!l\exists w',w''\!\leqslant\!1
$

\;
$([t]_p\!=\!\langle j,z,w'\rangle\wedge  [t]_q\!=\!\langle k,z,w''\rangle\wedge [w\!=\!1\leftrightarrow w'\!=\!0\vee w''\!=\!1])\big)\big] \vee $

$\big[\exists j\!<\!i\exists u,v\!<\!s\big({\tt trm}(u)\wedge {\tt var}(v)\wedge [s]_i\!=\!\ulcorner\!(\forall v\!\leqslant\!u)[s]_j\!\urcorner\wedge\exists x\!\leqslant\!\mathfrak{p}_u^{z^u+1}$

\;
$[{\tt val}(u,z,x)\wedge\forall r\!\leqslant\!x
\exists p\!<\!l\exists w'\!\leqslant\!1([t]_p\!=\!\langle j,z[r/\ulcorner\!v\!\urcorner],w' \rangle)]  \wedge$

\hfill
$[w\!=\!1\leftrightarrow \forall r\!\leqslant\!x
\exists p\!<\!l\exists w'\!\leqslant\!1([t]_p\!=\!\langle j,z[r/\ulcorner\!v\!\urcorner],1 \rangle]\big)\big]
 \Big)\Big].$

As a result, ${\tt Sat}_{\Delta_0}(x,y)$ is {\sc pr} too, since it be written  as

$\exists s\!\leqslant\!\mathfrak{p}_x^{(x+1)^2}\exists t\!\leqslant\!\mathfrak{p}_{x^2}^{2^{\mathfrak{p}_x^{(x+1)^2}}
\cdot 3^{\mathfrak{p}_x^{\mathfrak{p}_x^{(y+1)^2}}}\cdot 5}\Big[{\tt sat}_{\Delta_0}{\tt seq}(s,t)\wedge\ell ast(s)\!=\!x\wedge\ell ast(t)\!=\!\langle\ell en(s)\!-\!1,y,1\rangle\Big]$.

Note that we took $\neg,\rightarrow$ and  $\forall$ as the only logical connectives and  we coded $\langle i,z,w\rangle$ as $2^i\cdot 3^z\cdot 5^w$ which imply the desired {\sc pr} bounds.
\end{proof}

\section{Representability in Arithmetical Theories}\label{sec:rep}
A (most) natural definition for representability of a relation on the natural numbers  in a theory, whose language contains terms $\overline{n}$ indicating each natural number $n\!\in\!\mathbb{N}$,  is the following:

\begin{definition}[Weak Representability of Relations]\label{def:relwrep}{\rm

\

\noindent
A relation $R\subseteq\mathbb{N}$ is {\em weakly representable} in a theory $T$, if for some formula $\varphi(x)$ the equivalence
{$R(n)\iff T\vdash\varphi(\overline{n})$}
 holds for every $n\in\mathbb{N}$.
}\hfill\ding{71} \end{definition}

However, the following stronger definition is usually used in the literature on the incompleteness theorem:

\begin{definition}[Representability of Relations]\label{def:relrep}{\rm

\

\noindent
A relation $R\subseteq\mathbb{N}$ is  {\em representable} in a theory $T$, if for some formula $\varphi(x)$ the implications
 $R(n) \Longrightarrow T\vdash\varphi(\overline{n})$  and
 $\neg R(n)\Longrightarrow T\vdash\neg\varphi(\overline{n})$
 hold for every $n\in\mathbb{N}$.
}\hfill\ding{71}\end{definition}

Trivially, representability of a relation in a consistent theory implies its weaker representability in that theory. The converse does not hold, in the sense that a relation may be weakly representable in a theory without being representable (cf. \cite[Theorem II.2.16]{odifreddi}):

\begin{remark}[On the Representability of {\sc Provability}]\label{rem:prov}{\rm
\

\noindent
Let ${\rm Prov}_{\textsl{\textsf{Q}}}(x)$ be a provability predicate for Robinson Arithmetic $\textsl{\textsf{Q}}$; then for every sentence  $\phi$, we have  $\textsl{\textsf{Q}}\vdash\phi$ if and only if  $\textsl{\textsf{Q}}\vdash{\rm Prov}_{\textsl{\textsf{Q}}}(\ulcorner\phi\urcorner)$, since ${\rm Prov}_{\textsl{\textsf{Q}}}$ is a $\Sigma_1$-formula and $\textsl{\textsf{Q}}$ is $\Sigma_1$-complete and sound. On the other hand, there can be no formula $\Psi(x)$ such that for any formula $\phi$:

$^{_\bullet}\,$ if ${\rm Prov}_{\textsl{\textsf{Q}}}(\ulcorner\!\phi\!\urcorner)$, then $\textsl{\textsf{Q}}\vdash\Psi(\ulcorner\phi\urcorner)$;  and

$^{_\bullet}\,$ if $\neg{\rm Prov}_{\textsl{\textsf{Q}}}(\ulcorner\!\phi\!\urcorner)$, then $\textsl{\textsf{Q}}\vdash\neg\Psi(\ulcorner\phi\urcorner)$.

\noindent
Since, otherwise, provability in $\textsl{\textsf{Q}}$      would be decidable: for a given formula $\phi$ by running an exhaustive  proof search algorithm in $\textsl{\textsf{Q}}$ for the formulas $\Psi(\ulcorner\phi\urcorner)$ and $\neg\Psi(\ulcorner\phi\urcorner)$ in parallel, one could decide if $\textsl{\textsf{Q}}\vdash\phi$ (exactly when $\textsl{\textsf{Q}}\vdash\Psi(\ulcorner\phi\urcorner)$) or $\textsl{\textsf{Q}}\nvdash\phi$ (exactly when $\textsl{\textsf{Q}}\vdash\neg\Psi(\ulcorner\phi\urcorner)$) holds; and this is a contradiction (with Alonzo Church's Theorem).
}\hfill\ding{71}\end{remark}

For (total) functions we can have four different definitions for representability in theories (originated from \cite{tmr}).

\begin{definition}[Weakly Representable Functions]\label{def:wr}{\rm

\

\noindent
A function $f\colon\mathbb{N}\rightarrow\mathbb{N}$ is {\em weakly representable} in a theory $T$, if for some formula $\varphi(x,y)$ we have

$(1)\,$ if $f(n)=m$, then $T\vdash\varphi(\overline{n},\overline{m})$; and

$(2)\,$ if $f(n)\neq m$, then $T\not\vdash\varphi(\overline{n},\overline{m})$;

\noindent for every $n,m\in\mathbb{N}$.
}\hfill\ding{71}\end{definition}

\begin{definition}[Representable Functions]\label{def:r}{\rm

\

\noindent
A function $f\colon\mathbb{N}\rightarrow\mathbb{N}$ is {\em representable} in a theory $T$, if for some formula $\psi(x,y)$ we have

$(1)\,$ if $f(n)=m$, then $T\vdash\psi(\overline{n},\overline{m})$; and

$(2)\,$ if $f(n)\neq m$, then $T\vdash\neg\psi(\overline{n},\overline{m})$;

\noindent
for every $n,m\in\mathbb{N}$.
}\hfill\ding{71}\end{definition}

\begin{definition}[Strongly Representable Functions]\label{def:sr}{\rm

\

\noindent
A function $f\colon\mathbb{N}\rightarrow\mathbb{N}$ is {\em strongly representable} in a theory $T$, if for some formula $\theta(x,y)$ we have

$(1)\,$  $T\vdash\theta(\overline{n},\overline{f(n)})$; and

$(2)\,  T\vdash\forall y,z\big(\theta(\overline{n},y)
\wedge\theta(\overline{n},z)\rightarrow y=z\big)$;

\noindent for every $n\in\mathbb{N}$.
}\hfill\ding{71}\end{definition}

\begin{definition}[Provably Total Functions]\label{def:pt}{\rm

\

\noindent
A function $f\colon\mathbb{N}\rightarrow\mathbb{N}$ is {\em provably total} in a theory $T$, if for some formula $\eta(x,y)$ we have

$(1)\,$ $T\vdash\eta(\overline{n},\overline{f(n)})$; and

$(2)\,  T\vdash\forall x\exists y\big(\eta(x,y)\wedge\forall z\big[\eta(x,z)\rightarrow y=z\big]\big)$;

\noindent for every $n\in\mathbb{N}$.
}\hfill\ding{71}\end{definition}

Indeed, these definitions get stronger from top to bottom:
If $T$ is consistent and can prove $\overline{i}\!\neq\!\overline{j}$ for every distinct $i,j\!\in\!\mathbb{N}$, then every provably total function is strongly representable, and every strongly representable function is representable, and every representable function is weakly representable in $T$ with the same formula.
 It is a folklore result that representability implies strong representability (cf. \cite[Proposition I.3.3]{odifreddi}):

\begin{lemma}[Representability  $\Longrightarrow$ Strong Representability]\label{lem:rsr}

\

\noindent
In a theory $T$ which can prove the sentences  $\forall y(y\!<\!\overline{n} \,\vee\,   y\!=\!\overline{n} \,\vee\,  \overline{n}\!<\!y)$, $\forall y (y\not<0)$  and $\forall y (y\!<\!\overline{n\!+\!1} \leftrightarrow y\!=\!\overline{0} \,\vee\, \cdots \,\vee\, y\!=\!\overline{n})$, for all $n\!\in\!\mathbb{N}$, representability of a function implies its strong representability.
\end{lemma}
\begin{proof}
\

\noindent
If $f$ is representable by the formula $\psi(x,y)$ in $T$, then let $\theta(x,y)$ be   $\psi(x,y)\wedge\forall z\!<\!y\neg\psi(x,z)$. We now show that $T\vdash\theta(\overline{n},\overline{f(n)})$ and $T\vdash\theta(\overline{n},y)\rightarrow y=\overline{f(n)}$ hold for any $n\!\in\!\mathbb{N}$  as follows.  Reason in $T$:
If $z\!<\!\overline{f(n)}$, then if $f(n)\!=\!0$ we have a contradiction, otherwise (if $f(n)\!\neq\!0)$ we have  $z\!=\!\overline{i}$ for some $i\!<\!f(n)$. Of course for any such $i$ we have $\neg\psi(\overline{n},\overline{i})$; thus $\neg\psi(\overline{n},z)$.
If $\theta(\overline{n},y)$ and $y\!\neq\!\overline{f(n)}$, then either $y\!<\!\overline{f(n)}$ or $\overline{f(n)}\!<\!y$. In the former case we have $y\!=\!\overline{i}$ for some $i\!<\!f(n)$, if $f(n)\!\neq\!0$; otherwise $y\!<\!0$ is a contradiction, and so by   $\neg\psi(\overline{n},\overline{i})$ we have    $\neg\psi(\overline{n},y)$, which is a contradiction with $\theta(\overline{n},y)$. In the latter case, by $\forall z\!<\!y\neg\psi(\overline{n},z)$ we should have $\neg\psi(\overline{n},\overline{f(n)})$;  a contradiction again.
\end{proof}

\noindent
The question if the strong representability implies the provable totality was mentioned   open   in the first edition (1964) of the  classical book \cite{mendelson}. In 1965,   {\sc Verena Esther Huber-Dyson} showed  that  the strong representability implies the provable totality \cite{dyson}, and as a result this was Exercise~3.35 in the second edition (1979) of that book, and Exercise~3.32 in the third edition (1987), attributed to V.\, H.~Dyson. Then in the fourth (1997), the fifth (2009) and the sixth (2015) editions, this has been proved in Proposition~3.12, attributed to V.\, H.~Dyson again.

\begin{theorem}[Strong Representability $\Longrightarrow$ Provable Totality]\label{th:dyson}

\

\noindent
If a function is strongly representable in a theory, then it is provably total in that theory.
\end{theorem}
\begin{proof}
\

\noindent
Let us note that we do not put any condition on the  theory $T$; let $f$ be strongly representable by $\theta$ in   $T$.
Let $\exists ! u\, \mathcal{A}(u)$ be an abbreviation for the formula $\exists u \big(\mathcal{A}(u)\wedge\forall v [\mathcal{A}(v)\rightarrow v\!=\!u]\big)$. Put $$\eta(x,y)=\big[\exists ! z\,\theta(x,z)\wedge\theta(x,y)\big]\vee\big[\neg\exists ! z\,\theta(x,z)\wedge y\!=\!0\big].$$ For any $n\!\in\!\mathbb{N}$ we have $T\vdash\exists ! y\,\theta(\overline{n},y)$; thus from $T\vdash\theta(\overline{n},\overline{f(n)})$ we get $T\vdash\eta(\overline{n},\overline{f(n)})$. Now, we show that $T\vdash\forall x\exists ! y\, \eta(x,y)$. Reason inside $T$: If $\exists ! z\,\theta(x,z)$, then that unique $z$ which satisfies $\theta(x,z)$ also satisfies $\eta(x,z)$ and $\forall u\big[\eta(x,u)\rightarrow u\!=\!z\big]$, whence $\exists ! y \,\eta(x,y)$. If $\neg\exists ! z\,\theta(x,z)$, then  $y\!=\!0$ is the unique $y$ that satisfies  $\eta(x,y)$.
\end{proof}

The above proof of Dyson appears also in \cite[page 63]{jonessheph}, \cite[Proposition 3.8]{lambek} and \cite[Proposition 9.4.2]{su}. The following theorem 
is usually proved by showing that every weakly representable function is recursive and that every recursive function is (strongly) representable; see e.g. \cite[Corollary I.7.8]{odifreddi} or \cite[Theorem 4.5]{rautenberg}. Here we present a new proof.

\begin{theorem}[Weak Representability $\Longrightarrow$ Representability]\label{thm:main}

\

\noindent
For a theory  $T$, suppose the formula ${\rm Proof}_T(z,x)$ states that ``$z$ is (the G\"odel code of) the proof of a formula (with G\"odel code) $x$ in $T$'', and suppose that $T$ has the following properties:

\noindent
\textup{($\textsf{a}$)}\, $T\vdash\overline{i}\!\neq\!\overline{j}$ and $T\vdash\overline{n}\!\leqslant\!\overline{m}$ and $T\vdash\forall y(\overline{m}\!\leqslant\!y
\rightarrow\overline{n}\!\leqslant\!y)$, for any $i,j,n,m\!\in\!\mathbb{N}$ with $i\!\neq\!j$ and $n\!\leqslant\!m$;

\noindent
\textup{($\textsf{b}$)}\, $T\vdash \forall y (y\!\leqslant\!\overline{n} \,\vee\,  \overline{n}\!\leqslant\!y)$, for all $n\!\in\!\mathbb{N}$;

\noindent
\textup{($\textsf{c}$)}\, $T\vdash \forall y (y\!\leqslant\!\overline{n} \leftrightarrow \bigvee\!\!\!\!\!\bigvee_{i=0}^n y\!=\!\overline{i})$, for all $n\!\in\!\mathbb{N}$;

\noindent
\textup{($\textsf{d}$)}\, if $T\vdash\phi$ and $k$ is the G\"odel code of this proof, then   $T\vdash{\rm Proof}_T(\overline{k},\ulcorner\phi\urcorner)$;

\noindent
\textup{($\textsf{e}$)}\, if $k$ is not the G\"odel code of a proof of $\phi$ in $T$, then $T\vdash\neg{\rm Proof}_T(\overline{k},\ulcorner\phi\urcorner)$, in particular, if $T\not\vdash\phi$, then $T\vdash\neg{\rm Proof}_T(\overline{l},\ulcorner\phi\urcorner)$, for any $l\!\in\!\mathbb{N}$.

\noindent   Then weak representability of  a function   implies its representability in $T$.
\end{theorem}
\begin{proof}
\noindent

\noindent
Suppose the function $f$ is weakly representable by $\varphi$ in $T$. For the (bounded provability) predicate $\boldsymbol\varrho(z,x) = \exists u\!\leqslant\!z \,{\rm Proof}_T(u,x)$, let $\psi(x,y)=\exists z\big[\boldsymbol\varrho(z,\ulcorner\varphi(x,y)\urcorner)
\wedge\forall y'\!\leqslant\!z\, [y'\!\neq\!y\rightarrow\neg\boldsymbol\varrho
(z,\ulcorner\varphi(x,y')\urcorner)]
\big]$. For showing the representability of $f$ by $\psi$ in $T$ we prove that:

\qquad {\bf (1)} $T\vdash\psi(\overline{n},\overline{f(n)})$ for all $n\!\in\!\mathbb{N}$, and

\qquad {\bf (2)} $T\vdash\neg\psi(\overline{n},\overline{m})$ for all $n,m\!\in\!\mathbb{N}$ with $m\!\neq\!f(n)$.

\noindent {\bf (1)}: Fix an $n\!\in\!\mathbb{N}$ and let $k\!\in\!\mathbb{N}$ be a G\"odel code for the proof of $T\vdash\varphi(\overline{n},\overline{f(n)})$; so, we have  $f(n)\!\leqslant\!k$. By ($\textsf{d}$) above we have   $T\vdash{\rm Proof}_T(\overline{k},
\ulcorner\varphi(\overline{n},\overline{f(n)})\urcorner)$, and so $T\vdash\boldsymbol\varrho(\overline{k},
\ulcorner\varphi(\overline{n},\overline{f(n)})\urcorner)$ by ($\textsf{a}$) above. Now, for any $i\!\in\!\mathbb{N}$ with $i\!\neq\!f(n)$ we have that  $T\not\vdash\varphi(\overline{n},\overline{i})$, and so by ($\textsf{e}$) above, $T\vdash\neg{\rm Proof}_T(\overline{l},\ulcorner
\varphi(\overline{n},\overline{i})\urcorner)$ for any $l\!\in\!\mathbb{N}$. Thus,  by ($\textsf{c}$) above, $T\vdash\neg
\boldsymbol\varrho(\overline{l},\ulcorner
\varphi(\overline{n},\overline{i})\urcorner)$.
Reason in $T$: for any $y'$  with $y'\!\leqslant\!\overline{k}$ and $y'\!\neq\!\overline{f(n)}$, by  ($\textsf{c}$) above, we have $y'\!=\!\overline{j}$ for some $j\!\leqslant\!k$ with $j\!\neq\!f(n)$. For any such $j$ we have $\neg
\boldsymbol\varrho(\overline{k},\ulcorner
\varphi(\overline{n},\overline{j})\urcorner)$; and so, by ($\textsf{c}$) above, the sentence $\forall y'\!\leqslant\!\overline{k} \, [y'\!\neq\!y\rightarrow\neg\boldsymbol\varrho
(\overline{k},\ulcorner\varphi(\overline{n},y')\urcorner)]$ holds. Thus, $\psi(\overline{n},\overline{f(n)})$.

\noindent {\bf (2)}:  Fix some $n,m\!\in\!\mathbb{N}$ with $m\!\neq\!f(n)$.  Let us note that we already have:

\centerline{$\neg\psi(x,y)\!\equiv\!\forall   z\big[\boldsymbol\varrho(z,\ulcorner\varphi(x,y)\urcorner)
\rightarrow\exists y'\!\leqslant\!z\, [y'\!\neq\!y\wedge\boldsymbol\varrho
(z,\ulcorner\varphi(x,y')\urcorner)]
\big]$.}

\noindent For proving $T\vdash\neg\psi(\overline{n},\overline{m})$ we show that

\centerline{$T\vdash\forall   z\big[\boldsymbol\varrho(z,
\ulcorner\varphi(\overline{n},\overline{m})\urcorner)
\rightarrow \overline{f(n)}\!\leqslant\!z \wedge \overline{f(n)}\!\neq\!\overline{m}\wedge\boldsymbol\varrho
(z,\ulcorner\varphi(\overline{n},\overline{f(n)})\urcorner)
\big]$.}

\noindent Let $k\!\in\!\mathbb{N}$ be a G\"odel code for the proof of $T\vdash\varphi(\overline{n},\overline{f(n)})$; so, $f(n)\!\leqslant\!k$. Also, from $T\not\vdash\varphi(\overline{n},\overline{m})$, by ($\textsf{e}$) above, we have $T\vdash\neg\boldsymbol\varrho
(\overline{l},\varphi
(\ulcorner\overline{n},\overline{m})\urcorner)$, for any $l\!\in\!\mathbb{N}$. Reason in $T$: for any $z$, by ($\textsf{b}$) above, we have either {\bf (2.i)}~$z\!\leqslant\!\overline{k}$ or {\bf (2.ii)}~$\overline{k}\!\leqslant\!z$.
{\bf (2.i)}: If $z\!\leqslant\!\overline{k}$ then $z=\overline{i}$ for some $i\!\leqslant\!k$, by ($\textsf{c}$) above. Now,
$\boldsymbol\varrho(\overline{i},
\ulcorner\varphi(\overline{n},\overline{m})\urcorner)
\rightarrow \overline{f(n)}\!\leqslant\!\overline{i}\wedge \overline{f(n)}\!\neq\!\overline{m}\wedge\boldsymbol\varrho
(\overline{i},\ulcorner\varphi
(\overline{n},\overline{f(n)})\urcorner)$ follows from
$\neg\boldsymbol\varrho
(\overline{i},\ulcorner\varphi
(\overline{n},\overline{m})\urcorner)$;
thus $\neg\psi(\overline{n},\overline{m})$ holds.
{\bf (2.ii)}:  If $\overline{k}\!\leqslant\!z$, then $\overline{f(n)}\!\leqslant\!z$,  by ($\textsf{a}$) above, which also implies $\overline{f(n)}\!\neq\!\overline{m}$. On the other hand, we have  ${\rm Proof}_T(\overline{k},\ulcorner\varphi
(\overline{n},\overline{f(n)})\urcorner)$ and so $\exists u\!\leqslant\!z\,{\rm Proof}_T(u,\ulcorner\varphi
(\overline{n},\overline{f(n)})\urcorner)$, or equivalently $\boldsymbol\varrho
(z,\ulcorner\varphi
(\overline{n},\overline{f(n)})\urcorner)$. Thus,
$\neg\psi(\overline{n},\overline{m})$ holds since
we have $\overline{f(n)}\!\leqslant\!z \wedge \overline{f(n)}\!\neq\!\overline{m}\wedge\boldsymbol\varrho
(z,\ulcorner\varphi(\overline{n},\overline{f(n)})\urcorner)$.
\end{proof}

\noindent Let us note that
the (very weak) finitely axiomatizable
Robinson's Arithmetic $\textsf{\textsl{Q}}$
satisfies all the conditions ($\textsf{a}$,$\textsf{b}$,$\textsf{c}$,$\textsf{d}$,$\textsf{e}$) in Theorem~\ref{thm:main}.

\paragraph{Acknowledgements:}
\!\!\!\!\!The author was supported by   grant \textnumero \, \texttt{S}\! /\! \textsf{712} from the University of Tabriz, {\sc Iran}.

\bibliographystyle{wileynum}



\smallskip

\medskip

\bigskip

\noindent\rule{0.625\textwidth}{0.5pt}


\noindent
Research Institute for Fundamental Sciences,

\noindent
University of Tabriz, 29 Bahman Boulevard,

\noindent
P.O.Box 51666-16471, Tabriz, {\sc Iran}.

\noindent
\textsf{salehipour@tabrizu.ac.ir}

\vspace{-0.5em}

\noindent\rule{0.625\textwidth}{0.5pt}

\end{document}